\documentclass[11pt]{article}
\usepackage[margin=1in]{geometry}
\usepackage{graphicx}
\usepackage{amsmath}
\usepackage{amsthm}
\usepackage{physics}
\usepackage{amssymb}
\usepackage{xcolor}
\usepackage{mathtools}
\usepackage{comment}
\usepackage{subcaption}
\usepackage{enumitem}
\usepackage[normalem]{ulem}
\usepackage{tabularx}
\usepackage{booktabs}
\usepackage{rotating}
\usepackage{hyperref}
\usepackage{longtable}
\usepackage[user]{zref}
\usepackage{array}
\usepackage{ragged2e}
\usepackage{authblk}
\usepackage{multirow}

\hypersetup{
  colorlinks,
  linkcolor={red!50!black},
  citecolor={blue!50!black},
  urlcolor={blue!80!black}
}

\theoremstyle{plain}
\newtheorem*{theorem*}{Theorem}
\newtheorem{theorem}{Theorem}[section]
\newtheorem{lemma}[theorem]{Lemma}
\newtheorem{proposition}[theorem]{Proposition}
\newtheorem{corollary}[theorem]{Corollary}

\theoremstyle{definition}

\theoremstyle{remark}

\newtheorem*{claim}{Claim}

\numberwithin{equation}{section}

\usepackage{dsfont}
\newcommand{\indicator}[1]{\mathds{1}_{#1}}
\DeclarePairedDelimiter{\set}{\{}{\}}
\DeclarePairedDelimiter{\mset}{\langle}{\rangle}
\DeclarePairedDelimiter{\cbraces}{\{}{\}}
\DeclarePairedDelimiter{\brackets}{[}{]}

\DeclarePairedDelimiter{\floor}{\lfloor}{\rfloor}

\newcommand{\R}{\mathbb{R}}
\newcommand{\T}{\mathbb{T}}
\newcommand{\Z}{\mathbb{Z}}
\newcommand{\cS}{\mathcal{S}}
\newcommand{\cO}{\mathcal{O}}
\newcommand{\pmatch}[1]{\operatorname{pm}(#1)}
\newcommand{\per}{\operatorname{per}}
\newcommand{\Aut}{\operatorname{Aut}}

\let\emptyset\varnothing
\let\vec\mathbf

\allowdisplaybreaks

\title{A new upper bound on the dimer constant of $\mathbb{Z}^3$}
\author{Qidong He\thanks{Email address: \url{q.he@northeastern.edu}}}
\affil{Department of Mathematics, Northeastern University}
\date{}

\begin{document}
\maketitle

\begin{abstract}
  We prove that the dimer constant of the cubic lattice is bounded above by $0.452130$, improving the previous best upper bound of $0.457547$ obtained by Lundow (2001) using the transfer-matrix method of Ciucu (1998).
  Our construction instead uses diagonal transfer layers and an inequality of Csikv{\'a}ri (2017) to relate the spectral radius of the resulting matrix to the dimer constant, rather than exploiting symmetry special to the rectangular torus.
  We reduce the dimension of the transfer matrix using a compression method similar to Lundow's.
  Finally, we use a variant of the Collatz--Wielandt formula, derived from an asymptotic result of Friedland--Schneider (1980), to obtain a monotone sequence of upper bounds converging to the true spectral radius, each of which can be computed without forming the full matrix.
\end{abstract}

\section{Introduction}

The (simple) cubic lattice is the graph on the vertex set $\Z^3$ in which two vertices are connected by an edge if and only if they are at Euclidean distance one.
Given positive integers $a_1,a_2,a_3$ such that the product $a_1a_2a_3$ is even, let $f(a_1,a_2,a_3)$ denote the number of perfect matchings in the subgraph of $\Z^3$ induced by 
\begin{equation}
  \set{0\le x_i<a_i,\;i=1,2,3}\cap\Z^3.
\end{equation}
In 1966, Hammersley~\cite[Lemma~6]{hammersley1966existence} proved that the limit
\begin{equation}
  \ell_3:=\lim_{\substack{a_1,a_2,a_3\to\infty\\a_1a_2a_3\text{ even}}}\frac{1}{a_1a_2a_3}\log f(a_1,a_2,a_3)
\end{equation}
exists and is independent of the sequence taken.
We call $\ell_3$ the \textbf{dimer constant} of $\Z^3$.

Whereas the analogous constant for the square lattice is known in closed form~\cite{kasteleyn1961statistics,temperley1961dimer}, no such expression is known for the cubic lattice.
In 1966, Nagle~\cite{nagle1966new} obtained the estimate $\ell_3\approx0.446$ by a non-rigorous series expansion.
In 2018, Vanderstraeten--Vanhecke--Verstraete~\cite{vanderstraeten2018residual} reported the estimate $\ell_3\approx 0.44988452$ based on tensor-network methods.
More recently, in 2022, Otsuka--Akutsu~\cite{otsuka2022residual} estimated $\ell_3\approx 0.44715$ using tensor renormalization group techniques.
As of today, the best known rigorous bounds on $\ell_3$ are
\begin{equation}
  \frac{1}{2}\log\frac{5^5}{6^4}+0.001=0.441075\ldots \le \ell_3 \le 0.457547.
\end{equation}
The lower bound\footnote{In a private communication, Csikv\'ari noted to us that the number $0.001$ was not optimal and, according to a computation by GPT which was not human-verified, could be plausibly improved to at least $0.002775$.} was obtained by Csikv\'ari~\cite{csikvari2016matchings} in 2016 and constituted the first improvement since 1983 over the original bound of Minc~\cite{minc1983theory}, which followed from Schrijver's theorem~\cite{schrijver1998counting} (formerly a conjecture of Schrijver--Valiant~\cite{schrijver1980lower}) on the number of perfect matchings in regular bipartite graphs.
The upper bound was obtained by Lundow~\cite{lundow2001compression} in 2001, based on a transfer-matrix strategy originally devised by Ciucu~\cite{ciucu1998improved} in 1998.\footnote{In 2019, Chen--Golin--Yong~\cite{chen2019transfer} claimed that they had obtained a better upper bound on $\ell_3$ than Lundow's using a different method, but stopped short of reporting a numerical value.}
We refer the reader to a survey by Klivans--Saldanha~\cite{Klivans_Saldanha_2026} for recent developments on dimer coverings in dimension $d\ge 3$; see also the recent works of Taggi~\cite{taggi2022uniformly} on monomer-monomer correlations and Chandgotia--Sheffield--Wolfram~\cite{chandgotia2023large} on large deviations.

The main result of this paper is the following.

\begin{theorem}
  \label{thm: main}
  $\ell_3\le 0.452130$.
\end{theorem}

To explain the source of the improvement, we first recall the transfer-matrix construction underlying the previous upper bound.
There, a toroidal box is decomposed into $m\times n$ cross sections parallel to a coordinate plane, with a transfer state recording the vertices in one cross section which are matched into the next.
Counting the number of perfect matchings in a toroidal box in two coordinate directions as in~\cite[Lemma~2.2]{ciucu1998improved}, relying \textit{crucially} on the box's symmetries, bounds $\ell_3$ in terms of the spectral radius of the resulting $2^{mn}\times 2^{mn}$ sparse matrix; since this comparison uses even powers, the method requires both $m$ and $n$ to be \textit{even}.
The computation is then made feasible by compressing the matrix under a suitable symmetry group, allowing Ciucu to treat the case $(m,n)=(4,4)$ and Lundow the case $(m,n)=(4,6)$.
Further upscaling remains computationally prohibitive; see also the discussions in Friedland--Peled~\cite{friedland2005theory} and Lundow--Markstr{\"o}m~\cite{lundow2008exact}.

While our construction is also based on the transfer-matrix method, it differs from Ciucu's in both the setup and the comparison argument.
We take the toroidal cross sections in the diagonal planes $x_1+x_2+x_3=h$, $h\in\Z$, while retaining Ciucu's definition of the transfer states.
In this setup, every edge of the cubic lattice connects vertices in consecutive cross sections, and the entries of the transfer matrix are given by permanents of submatrices of a fixed biadjacency matrix. 
The comparison of the spectral radius of the transfer matrix with $\ell_3$ is instead accomplished using an inequality relating the matching generating function of a bipartite graph to those of its $2$-lifts, yielding an upper bound for all transfer layer dimensions $m,n\ge 2$ without any parity restriction.

Even before exploiting this additional flexibility, our method already produces a better bound at $(m,n)=(4,3)$ than Lundow's.
Our transfer matrix can moreover be compressed by grouping states into orbits under the action of a suitable symmetry group.
The group we identify differs from Ciucu's in that its action on the states does not commute with the transfer matrix; cf.~\cite[Lemma~3.1]{ciucu1998improved}.
Instead, we prove that the orbit partition satisfies \textit{property C} in the sense of Lundow~\cite[Definition~1]{lundow2001compression} and show that the resulting compression preserves the spectral radius without the nonzero spectral gap assumption in~\cite[Theorem~11]{lundow2001compression}.

Finally, we establish an iterated version of the Collatz--Wielandt formula~\cite[Theorem~8.1.26]{horn2012matrix} for nonnegative matrices with positive diagonal; its proof, given in Appendix~\ref{appx: iterated CW}, makes use of the precise entrywise asymptotics for powers of a nonnegative matrix due to Friedland--Schneider~\cite{friedland1980growth}. 
Applied to the compressed transfer matrix, this yields an equivalent family of bounds that can be evaluated without accessing the full matrix, allowing for the use of larger transfer layers at the cost of longer computation time.

\section{A family of upper bounds}
\label{sec: upper bounds}

Let $m,n\ge 2$ be integers.
Write $V_{m,n}:=\Z/m\Z\times\Z/n\Z$, with addition understood component-wise.
Define the matrix $A_{m,n}$, whose rows and columns are indexed by $V_{m,n}$, by
\begin{equation}
  \label{eqn: biadjacency matrix}
  \bra{\vec{u}}A_{m,n}\ket{\vec{v}}
  :=\indicator{\set{(0,0),(1,0),(0,1)}}(\vec{v}-\vec{u}).
\end{equation}
For $R,C\subseteq V_{m,n}$, write $A_{m,n}[R,C]$ for the submatrix of $A_{m,n}$ with row set $R$ and column set $C$.
Define the matrix $T_{m,n}$, whose rows and columns are indexed by the power set of $V_{m,n}$, by
\begin{equation}
  \label{eqn: transfer matrix}
  \bra{U}T_{m,n}\ket{V}
  :=\begin{cases}
    \per A_{m,n}[U,V_{m,n}\setminus V] & \text{if }\abs{U}+\abs{V}=mn\\
    0 & \text{otherwise}
  \end{cases},
\end{equation}
where $\per(\cdot)$ denotes the permanent, with the convention that the permanent of the empty matrix is $1$.
Let $\rho(T_{m,n})$ be the spectral radius of $T_{m,n}$.
The main result of this section is the following family of upper bounds on $\ell_3$; notably, there is no parity restriction on $m$ and $n$.

\begin{theorem}
  \label{thm: bound}
  For all $m,n\ge 2$, $\ell_3\le\frac{1}{mn}\log\rho(T_{m,n})$.
\end{theorem}

The rest of this section is devoted to the proof of Theorem~\ref{thm: bound}.

Let $\vec{e}_1,\vec{e}_2,\vec{e}_3$ denote the standard basis vectors of $\Z^3$.
For even $l\ge 4$, define the lattice
\begin{equation}
  \Lambda_{m,n,l}:=\mset[\big]{m(\vec{e}_1-\vec{e}_3),n(\vec{e}_2-\vec{e}_3),l\vec{e}_3}
\end{equation}
and let
\begin{equation}
  \T_{m,n,l}:=\Z^3/\Lambda_{m,n,l}
\end{equation}
be the corresponding quotient graph. 
Thus, the vertices of $\T_{m,n,l}$ are the cosets $[\vec{x}]$, with two cosets
adjacent if and only if they contain adjacent representatives.
Under the assumptions on $m,n,l$, no nonzero difference between two elements of $\set{\vec{0},\pm\vec{e}_1,\pm\vec{e}_2,\pm\vec{e}_3}$ belongs to $\Lambda_{m,n,l}$.
Hence, the six neighbors $[\vec{x}\pm\vec{e}_i]$, $i=1,2,3$, of every vertex $[\vec{x}]$ of $\T_{m,n,l}$ are distinct and none coincides with $[\vec{x}]$, so that $\T_{m,n,l}$ is \textit{simple} and $6$-regular.
Moreover, every element of $\Lambda_{m,n,l}$ has even coordinate sum, so the parity of the coordinate sum descends to the quotient, making $\T_{m,n,l}$ bipartite.
We orient each edge $\set{[\vec{x}],[\vec{x}+\vec{e}_i]}$ of $\T_{m,n,l}$, $i=1,2,3$, from $[\vec{x}]$ to $[\vec{x}+\vec{e}_i]$.

For a finite graph $G$, let $\pmatch{G}$ denote the number of perfect matchings in $G$.

\begin{lemma}
  \label{lem: transfer}
  For all $m,n\ge 2$ and even $l\ge 4$, $\pmatch{\T_{m,n,l}}=\tr(T_{m,n}^l)$.
\end{lemma}

\begin{proof}
  Consider the discrete diagonal planes $x_1+x_2+x_3=h$ in $\Z^3$, where $h\in\Z$ is interpreted as the \textbf{height} of the plane.
  The first two generators of $\Lambda_{m,n,l}$ preserve the height, while the third generator increases the height by $l$. 
  Hence, taking the quotient by $\Lambda_{m,n,l}$ imposes periodicity within each diagonal plane and identifies diagonal planes whose heights differ by $l$; formally, we have the group isomorphism
  \begin{equation}
    \begin{split}
      \Phi:\T_{m,n,l}=\Z^3/\Lambda_{m,n,l} {}& \to V_{m,n}\times\Z/l\Z
      \\
      [\vec{x}]=[(x_1,x_2,x_3)] {}& \mapsto ((x_1\bmod m,x_2\bmod n),x_1+x_2+x_3\bmod l).
    \end{split}
  \end{equation}
  For $k\in\Z/l\Z$, call $L_k:=\Phi^{-1}(V_{m,n}\times\set{k})$ the diagonal \textbf{layer} of height $k$.
  Observe that if $\Phi([\vec{x}])=(\vec{u},k)$, then
  \begin{equation}
    \Phi([\vec{x}+\vec{e}_1])=(\vec{u}+(1,0),k+1),\quad
    \Phi([\vec{x}+\vec{e}_2])=(\vec{u}+(0,1),k+1),\quad
    \Phi([\vec{x}+\vec{e}_3])=(\vec{u},k+1).
  \end{equation}
  Thus, every oriented edge of $\T_{m,n,l}$ goes from $L_k$ to $L_{k+1}$ for some $k\in\Z/l\Z$, and the biadjacency matrix of the edges between these two layers is exactly $A_{m,n}$.

  Now, given a perfect matching $M$ in $\T_{m,n,l}$, define its \textbf{state} at height $k\in\Z/l\Z$ as
  \begin{equation}
    S_k(M):=\set{\vec{u}\in V_{m,n}\mid\Phi^{-1}(\vec{u},k)\text{ is matched in $M$ along an oriented edge from $L_k$ to $L_{k+1}$}}.
  \end{equation}
  Consequently, given $k\in\Z/l\Z$ and states $U,V\subseteq V_{m,n}$, a perfect matching $M$ in $\T_{m,n,l}$ such that $S_k(M)=U$ and $S_{k+1}(M)=V$ matches the vertices of $L_k$ corresponding to $U$ onto those of $L_{k+1}$ corresponding to $V_{m,n}\setminus V$.
  Such a matching exists only if $\abs{U}+\abs{V}=mn$, and, in that case, the number of these matchings between the two layers is given by $\bra{U}T_{m,n}\ket{V}=\per A_{m,n}[U,V_{m,n}\setminus V]$.
  The proof is complete by application of the transfer matrix method.
\end{proof}

Next, we compare the number of perfect matchings in quotient graphs $\T_{m,n,l}$ of different periods.
Recall that a graph $H$ is a \textbf{$2$-lift} of a graph $G$ if there exists a $2$-to-$1$ map $\pi:V(H)\to V(G)$ which maps the neighborhood of each vertex of $H$ bijectively onto the neighborhood of its image.

\begin{lemma}
  For all $m,n\ge 2$ and even $l\ge 4$, the graphs $\T_{2m,n,l}$, $\T_{m,2n,l}$ and $\T_{m,n,2l}$ are $2$-lifts of $\T_{m,n,l}$.
\end{lemma}

\begin{proof}
  We treat $\T_{2m,n,l}$; the other cases are analogous.
  The group $\Lambda_{2m,n,l}$ is an index-$2$ subgroup of $\Lambda_{m,n,l}$, so the identity map on $\Z^3$ induces a $2$-to-$1$ map
  \begin{equation}
    \begin{split}
      \pi:\T_{2m,n,l} {}&\to \T_{m,n,l}
      \\
      [\vec{x}]_{\Lambda_{2m,n,l}} {}&\mapsto [\vec{x}]_{\Lambda_{m,n,l}}
    \end{split}
  \end{equation}
  which maps the neighborhood of each vertex of $\T_{2m,n,l}$ bijectively onto the neighborhood of its image.
  Having constructed the required $2$-to-$1$ map, we conclude that $\T_{2m,n,l}$ is a $2$-lift of $\T_{m,n,l}$.
\end{proof}

For a finite graph $G$, let $m_k(G)$ denote the number of matchings in $G$ with $k$ edges, and define its \textbf{matching generating function} by
\begin{equation}
  M(G,t):=\sum_{k=0}^{\floor{\abs{V(G)}/2}} m_k(G)t^k.
\end{equation}

\begin{theorem}[{Csikv{\'a}ri,~\cite[Lemma~4.2]{csikvari2017lower}}]
  \label{thm: 2-lift}
  If $G$ is a bipartite graph and $H$ is a $2$-lift of $G$, then, for all $t\ge 0$,
  \begin{equation}
    M(H,t)\le M(G,t)^2.
  \end{equation}
\end{theorem}

This result has been extended by Lelarge~\cite[Proposition~8]{lelarge2015counting} to the case of arbitrary positive weights on the edges; see also Barvinok~\cite[Theorem~5.3.3]{barvinok2016combinatorics}.
Here, we only need its specialization to perfect matchings:

\begin{corollary}
  \label{cor: 2-lift}
  If $G$ is a bipartite graph and $H$ is a $2$-lift of $G$, then $\pmatch{H}\le\pmatch{G}^2$.
\end{corollary}

\begin{proof}
  Using $\abs{V(H)}=2\abs{V(G)}$ and Theorem~\ref{thm: 2-lift}, we have
  \begin{equation}
    \pmatch{H}
    =\lim_{t\to\infty}\frac{M(H,t)}{t^{\abs{V(H)}/2}}
    =\lim_{t\to\infty}\frac{M(H,t)}{t^{\abs{V(G)}}}
    \le \lim_{t\to\infty}\frac{M(G,t)^2}{t^{\abs{V(G)}}}
    =\pmatch{G}^2,
  \end{equation}
  as required.
\end{proof}

We are now ready to deduce Theorem~\ref{thm: bound}.

\begin{proof}[Proof of Theorem~\ref{thm: bound}]
  Let $p\ge 2$ be an integer.
  Observe that the lattice
  \begin{equation}
    \Lambda_{2^{p}m,2^{p}n,2^{p}}
    =\mset[\big]{2^{p}m(\vec{e}_1-\vec{e}_3),2^{p}n(\vec{e}_2-\vec{e}_3),2^{p}\vec{e}_3}
    =\mset[\big]{2^{p}m\vec{e}_1,2^{p}n\vec{e}_2,2^{p}\vec{e}_3}
  \end{equation}
  is rectangular, so that
  \begin{equation}
    \label{eqn: rectangular}
    f(2^p m,2^p n,2^p)\le \pmatch{\T_{2^p m,2^p n,2^p}}.
  \end{equation}
  We now bound
  \begin{align}
    \frac{1}{2^{3p}mn}\log f(2^p m,2^p n,2^p)
    {}&\le \frac{1}{2^{3p}mn}\log\pmatch{\T_{2^p m,2^p n,2^p}}
    {}&& \text{by~\eqref{eqn: rectangular}}
    \nonumber
    \\
    {}&\le \frac{1}{2^{p}mn}\log\pmatch{\T_{m,n,2^p}}
    {}&& \text{by Corollary~\ref{cor: 2-lift}}
    \nonumber
    \\
    {}&= \frac{1}{2^{p}mn}\log\tr(T_{m,n}^{2^p})
    {}&& \text{by Lemma~\ref{lem: transfer}}
    \nonumber
    \\
    {}&\le \frac{1}{2^{p}mn}\log \brackets[\Big]{2^{mn}\rho(T_{m,n})^{2^p}},
    {}&& 
  \end{align}
  having used that the trace of a square matrix equals the sum of its eigenvalues with their algebraic multiplicities~\cite[pp.~50--51]{horn2012matrix} in the final inequality.
  Letting $p\to\infty$ completes the proof.
\end{proof}

\section{Computational reductions}

For $m,n\ge 2$, the transfer matrix $T_{m,n}$ has $2^{2mn}$ entries. 
While most of them vanish, even storing only the nonzero ones becomes prohibitive with larger $m,n$.
It would therefore be desirable to compute $\rho(T_{m,n})$ by considering only a smaller matrix and, even still, without storing any matrix at all.
In this section, we explain how to achieve these reductions.

\subsection{Compression by symmetry}

Write $D:=\set{(0,0),(1,0),(0,1)}\subset V_{m,n}$ (see~\eqref{eqn: biadjacency matrix}).
Let 
\begin{equation}
  \label{eqn: point group}
  K_{m,n} := \set{\varphi\in\Aut(V_{m,n})\mid \text{there exists }\vec{t}_\varphi\in V_{m,n}\text{ such that }\varphi(D)=\vec{t}_\varphi+D};
\end{equation}
thus, up to a possible translation, each transformation in $K_{m,n}$ permutes the elements of $D$.
Let
\begin{equation}
  G_{m,n}:=V_{m,n}\rtimes K_{m,n},
\end{equation}
where $K_{m,n}$ acts on $V_{m,n}$ in the natural way.
Let $G_{m,n}$ act on $V_{m,n}$ via
\begin{equation}
  \label{eqn: group action}
  (\vec{u},\varphi)\cdot\vec{v}:=\vec{u}+\varphi(\vec{v})
\end{equation}
and hence on its subsets.

\begin{proposition}
  \label{prop: property C}
  The partition of $2^{V_{m,n}}$ into orbits under the action of $G_{m,n}$ has \emph{property C} in the sense of Lundow~\cite[Definition~1]{lundow2001compression}:
  if $\cO_1,\cO_2\in 2^{V_{m,n}}/G_{m,n}$ and $U_1,U_2\in\cO_1$, then
  \begin{equation}
    \label{eqn: property C}
    \sum_{V\in\cO_2}\bra{U_1}T_{m,n}\ket{V}
    =\sum_{V\in\cO_2}\bra{U_2}T_{m,n}\ket{V}.
  \end{equation}
\end{proposition}

\begin{proof}
  For $g=(\vec{u},\varphi)\in G_{m,n}$, define $g_\ast:=(\vec{u}-\vec{t}_\varphi,\varphi)\in G_{m,n}$.
  Let $U,V\subseteq V_{m,n}$ be such that $\abs{U}+\abs{V}=mn$.
  For $g\in G_{m,n}$ and a bijection $\sigma:U\to V_{m,n}\setminus V$, define $\sigma_g:=g_\ast\circ\sigma\circ g^{-1}$.
  Thus, $\sigma_g$ is a bijection from $g\cdot U$ to $g_\ast\cdot(V_{m,n}\setminus V)$; moreover, the mapping $\sigma\mapsto\sigma_g$ is itself a bijection, with inverse $\tau\mapsto g_\ast^{-1}\circ\tau\circ g$.
  
  \begin{claim}
    For all $U,V\subseteq V_{m,n}$ and $g\in G_{m,n}$, $\bra{U}T_{m,n}\ket{V}=\bra{g\cdot U}T_{m,n}\ket{g_\ast\cdot V}$.
  \end{claim}

  \begin{proof}
    If $\abs{U}+\abs{V}\ne mn$, then $\bra{U}T_{m,n}\ket{V}=\bra{g\cdot U}T_{m,n}\ket{g_\ast\cdot V}=0$ by~\eqref{eqn: transfer matrix}.
    Otherwise, write $g=(\vec{w},\varphi)$, and observe that, for $\vec{u}\in U$ and a bijection $\sigma:U\to V_{m,n}\setminus V$,
    \begin{align}
      \label{eqn: matrix element under group action}
      \bra{g\cdot\vec{u}}A_{m,n}\ket{g_\ast\cdot\sigma(\vec{u})}
      {}&=\indicator{D}\big(g_\ast\cdot\sigma(\vec{u})-g\cdot\vec{u}\big)
      {}&&\text{by~\eqref{eqn: biadjacency matrix}}
      \nonumber
      \\
      {}&=\indicator{D}\big(\varphi(\sigma(\vec{u})-\vec{u})-\vec{t}_\varphi\big)
      {}&&\text{by~\eqref{eqn: group action}}
      \nonumber
      \\
      {}&=\indicator{\varphi(D)}\big(\varphi(\sigma(\vec{u})-\vec{u})\big)
      {}&&\text{by~\eqref{eqn: point group}}
      \nonumber
      \\
      {}&=\indicator{D}\big(\sigma(\vec{u})-\vec{u}\big)
      {}&&\text{since $\varphi$ is a bijection}
      \nonumber
      \\
      {}&=\bra{\vec{u}}A_{m,n}\ket{\sigma(\vec{u})}.
      {}&&
    \end{align}
    Therefore,
    \begin{align}
      \bra{U}T_{m,n}\ket{V}
      {}&=\per A_{m,n}[U,V_{m,n}\setminus V]
      {}&& \text{by~\eqref{eqn: transfer matrix}}
      \nonumber
      \\
      {}&=\sum_{\substack{\sigma:U\to V_{m,n}\setminus V\\\text{bijection}}}\prod_{\vec{u}\in U}\bra{\vec{u}}A_{m,n}\ket{\sigma(\vec{u})}
      {}&& \text{by the definition of the permanent}
      \nonumber
      \\
      {}&=\sum_{\substack{\sigma:U\to V_{m,n}\setminus V\\\text{bijection}}}\prod_{\vec{u}\in U}\bra{g\cdot\vec{u}}A_{m,n}\ket{\sigma_g(g\cdot\vec{u})}
      {}&&\text{by~\eqref{eqn: matrix element under group action} and using $g_\ast\cdot\sigma(\vec{u})=\sigma_g(g\cdot\vec{u})$}
      \nonumber
      \\
      {}&=\sum_{\substack{\sigma':g\cdot U\to V_{m,n}\setminus g_\ast\cdot V\\\text{bijection}}}\prod_{\vec{u}'\in g\cdot U}\bra{\vec{u}'}A_{m,n}\ket{\sigma'(\vec{u}')}
      {}&&\text{since $\sigma\mapsto\sigma_g$ is a bijection}
      \nonumber
      \\
      {}&=\bra{g\cdot U}T_{m,n}\ket{g_\ast\cdot V},
      {}&&
    \end{align}
    as claimed.
  \end{proof}

  It is now straightforward to verify~\eqref{eqn: property C}.
  Let $g\in G_{m,n}$ be such that $U_2=g\cdot U_1$.
  Then,
  \begin{align}
    \sum_{V\in\cO_2}\bra{U_1}T_{m,n}\ket{V}
    {}&=\sum_{V\in\cO_2}\bra{g\cdot U_1}T_{m,n}\ket{g_\ast\cdot V}
    {}&& \text{by the claim}
    \nonumber
    \\
    {}&=\sum_{V\in\cO_2}\bra{g\cdot U_1}T_{m,n}\ket{V}
    {}&& \text{since $g_\ast\cdot\cO_2=\cO_2$}
    \nonumber
    \\
    {}&=\sum_{V\in\cO_2}\bra{U_2}T_{m,n}\ket{V},
    {}&&
  \end{align}
  as required.
\end{proof}

Define the $G_{m,n}$-compressed version of $T_{m,n}$ to be the matrix $\widehat{T}_{m,n}$ whose rows and columns are indexed by $2^{V_{m,n}}/G_{m,n}$, with
\begin{equation}
  \bra{\cO_1}\widehat{T}_{m,n}\ket{\cO_2}
  :=\sum_{V\in\cO_2}\bra{U}T_{m,n}\ket{V},
  \quad\text{for any }U\in\cO_1;
\end{equation}
the choice of $U\in\cO_1$ is irrelevant by Proposition~\ref{prop: property C}.

We now show that the compression preserves the spectral radius.
Our result strengthens both~\cite[Lemma~3.2]{ciucu1998improved} which required the group action to commute with the uncompressed matrix and~\cite[Theorem~11]{lundow2001compression} which required a nonzero spectral gap.

\begin{proposition}
  \label{prop: compression preserves spectral radius}
  For all $m,n\ge 2$,
  \begin{equation}
    \label{eqn: compression preserves spectral radius}
    \rho(T_{m,n})=\rho\big(\widehat{T}_{m,n}\big).
  \end{equation}
\end{proposition}

\begin{proof}
  Define the matrix $P_{m,n}$ by 
  \begin{equation}
    \label{eqn: P}
    \bra{U}P_{m,n}\ket{[V]}
    :=\indicator{[V]}(U).
  \end{equation}
  By a straightforward computation, making use of Proposition~\ref{prop: property C}, we have that
  \begin{equation}
    \label{eqn: compression intertwining}
    T_{m,n}P_{m,n}=P_{m,n}\widehat{T}_{m,n}.
  \end{equation}
  Let $\norm{\cdot}_\infty$ denote the maximum row sum norm for square matrices~\cite[Ex.~5.6.5]{horn2012matrix}.
  Then, for all $q\ge 1$,
  \begin{align}
    \norm{T_{m,n}^q}_\infty
    {}&=\max_{U\subseteq V_{m,n}} \bra{U} T_{m,n}^q \bigg(\sum_{V\subseteq V_{m,n}} \ket{V}\bigg)
    {}&&\text{by the definition of $\norm{\cdot}_\infty$}
    \nonumber
    \\
    {}&=\max_{U\subseteq V_{m,n}} \bra{U} T_{m,n}^q P_{m,n} \bigg(\sum_{[V]\in 2^{V_{m,n}}/G_{m,n}} \ket{[V]}\bigg)
    {}&&\text{by~\eqref{eqn: P}}
    \nonumber
    \\
    {}&=\max_{U\subseteq V_{m,n}} \bra{U} P_{m,n} \widehat{T}_{m,n}^q \bigg(\sum_{[V]\in 2^{V_{m,n}}/G_{m,n}} \ket{[V]}\bigg)
    {}&&\text{by repeated use of~\eqref{eqn: compression intertwining}}
    \nonumber
    \\
    {}&=\max_{[U]\in 2^{V_{m,n}}/G_{m,n}} \bra{[U]} \widehat{T}_{m,n}^q \bigg(\sum_{[V]\in 2^{V_{m,n}}/G_{m,n}} \ket{[V]}\bigg)
    {}&&\text{by~\eqref{eqn: P}}
    \nonumber
    \\
    {}&=\norm{\widehat{T}_{m,n}^q}_\infty.
    {}&&
  \end{align}
  By Gelfand's formula~\cite[Corollary~5.6.14]{horn2012matrix}, we conclude that
  \begin{equation}
    \rho(T_{m,n})
    =\lim_{q\to\infty}\norm{T_{m,n}^q}_\infty^{1/q}
    =\lim_{q\to\infty}\norm{\widehat{T}_{m,n}^q}_\infty^{1/q}
    =\rho\big(\widehat{T}_{m,n}\big),
  \end{equation}
  as required.
\end{proof}

Lastly, we compute $K_{m,n}$.
Since every automorphism of $V_{m,n}$ fixes the identity $(0,0)$, there are only three possible choices of $\vec{t}_\varphi$ for $\varphi\in K_{m,n}$: $(0,0)$, $(-1,0)$, and $(0,-1)$.
Within each possibility, we have two additional choices for $\varphi((1,0))$; once that is fixed, $\varphi((0,1))$ is forced, and so is $\varphi$ by linear extension, though the resulting map may not be an automorphism of $V_{m,n}$ unless, for instance, the orders of the generators $(1,0)$ and $(0,1)$ are preserved.
In Table~\ref{tab: varphi}, we record the possible maps $\varphi$ and the condition for each to be an element of $\Aut(V_{m,n})$.
Thus,
\begin{equation}
  K_{m,n} \cong 
  \begin{cases}
    S_3 & \text{if }m=n \\
    C_2 & \text{if }m\ne n\text{ and }(m\mid n\text{ or }n\mid m) \\
    \set{1} & \text{if }m\nmid n\text{ and }n\nmid m
  \end{cases}
  .
\end{equation}

\begin{table}[t]
  \centering
  \begin{tabular}{ccccc}
    \toprule
    $\vec{t}_\varphi$ & $\varphi((1,0))$ & $\varphi((0,1))$ & $\varphi((x,y))$  & Condition for $\varphi\in\Aut(V_{m,n})$ \\ 
    \midrule
    \multirow{2}{*}{$(0,0)$} & $(1,0)$ & $(0,1)$ & $(x,y)$ & always \\
     & $(0,1)$ & $(1,0)$ & $(y,x)$ & $m=n$ \\
    \multirow{2}{*}{$(-1,0)$} & $(-1,0)$ & $(-1,1)$ & $(-x-y,y)$ & $m\mid n$ \\
     & $(-1,1)$ & $(-1,0)$ & $(-x-y,x)$ & $m=n$ \\
    \multirow{2}{*}{$(0,-1)$} & $(1,-1)$ & $(0,-1)$ & $(x,-x-y)$ & $n\mid m$ \\
     & $(0,-1)$ & $(1,-1)$ & $(y,-x-y)$ & $m=n$ \\
    \bottomrule
  \end{tabular}
  \caption{Possible maps $\varphi\in K_{m,n}$ and their admissibility conditions}
  \label{tab: varphi}
\end{table}

\subsection{Matrix-free and parallelizable implementation}
\label{sec: matrix-free implementation}

In this subsection, we show how to approximate $\rho(\widehat{T}_{m,n})$ from above and to arbitrary precision by a monotone sequence evaluated through repeated matrix-vector products, without needing to store the full matrix $\widehat{T}_{m,n}$, while also allowing for parallel computation.

The main ingredient is the following result, which is likely known to empiricists, but we have not found a formal reference.
For completeness, we give a proof in Appendix~\ref{appx: iterated CW} based on the ideas of Friedland--Schneider~\cite{friedland1980growth}, who worked under stronger assumptions and drew a stronger conclusion; specifically, out of the two equalities in~\cite[(6.2)]{friedland1980growth}, the second one alone suffices for our purposes.

\begin{proposition}[Iterated Collatz--Wielandt formula]
  \label{prop: iterated CW}
  Let $A$ be an $n\times n$, nonnegative matrix with positive diagonal and $\vec{x}\in\R^n$ be a positive vector.
  For $q\ge 0$, define
  \begin{equation}
    \label{eqn: definition of C_q}
    C_q:=\max_{1\le i\le n}\frac{(A^{q+1}\vec{x})_i}{(A^q\vec{x})_i}.
  \end{equation}
  Then, $(C_q)_{q\ge 0}$ is non-increasing, and
  \begin{equation}
    \rho(A)=\lim_{q\to\infty}C_q=\inf_{q\ge0}C_q.
  \end{equation}
\end{proposition}

Observe that $\bra{[U]}\widehat{T}_{m,n}\ket{[V]}$ vanishes unless $\abs{U}+\abs{V}=mn$.
For $0\le i\le\floor{mn/2}$, let
\begin{equation}
  \cS_i^-:=\set[\big]{[U]\in 2^{V_{m,n}}/G_{m,n}\mid \abs{U}=i},\quad
  \cS_i^+:=\set[\big]{[U]\in 2^{V_{m,n}}/G_{m,n}\mid \abs{U}=mn-i},
\end{equation}
and denote by $E_i^\pm$ the corresponding coordinate subspaces.
Since set complementation commutes with the action of $G_{m,n}$, we have the bijection
\begin{equation}
  \cS_i^+ \to \cS_i^-,\quad
  [U] \mapsto [V_{m,n}\setminus U]
\end{equation}
and hence $\abs{\cS_i^+}=\abs{\cS_i^-}$.
Moreover, defining the restrictions
\begin{equation}
  X_{m,n}^{(i)}:=\eval{\widehat{T}_{m,n}}_{E_i^+}: E_i^+ \to E_i^-,\quad
  Y_{m,n}^{(i)}:=\eval{\widehat{T}_{m,n}}_{E_i^-}: E_i^- \to E_i^+,
\end{equation}
we have the decomposition
\begin{equation}
  \widehat{T}_{m,n}^2 = 
  \begin{cases}
    \bigoplus_{i=0}^{mn/2-1} \big(X_{m,n}^{(i)}Y_{m,n}^{(i)} \oplus Y_{m,n}^{(i)}X_{m,n}^{(i)}\big) \oplus \big(X_{m,n}^{(mn/2)}Y_{m,n}^{(mn/2)}\big) & \text{if $mn$ is even} \\
    \bigoplus_{i=0}^{\floor{mn/2}} \big(X_{m,n}^{(i)}Y_{m,n}^{(i)} \oplus Y_{m,n}^{(i)}X_{m,n}^{(i)}\big) & \text{otherwise}
  \end{cases}
\end{equation}
and so 
\begin{equation}
  \label{eqn: spectral radius decomposition}
  \rho\big(\widehat{T}_{m,n}\big)^2
  =\rho\big(\widehat{T}_{m,n}^2\big)
  =\max_{0\le i\le\floor{mn/2}} \rho\big(X_{m,n}^{(i)}Y_{m,n}^{(i)}\big),
\end{equation}
having used that $\rho\big(X_{m,n}^{(i)}Y_{m,n}^{(i)}\big)=\rho\big(Y_{m,n}^{(i)}X_{m,n}^{(i)}\big)$ for each $i$~\cite[Problem~1.2.P17]{horn2012matrix}.

\begin{corollary}
  \label{cor: CW corollary}
  Let $m,n\ge 2$ be integers.
  For $0\le i\le \floor{mn/2}$, let $\ket{\vec{1}_i}$ denote the all-ones vector in the basis indexed by $\cS_i^-$.
  For $p\ge 0$, define
  \begin{equation}
    \label{eqn: CW upper bound}
    C_{m,n;p}^{(i)}:=\frac{1}{2mn}\log\cbraces[\bigg]{\max_{[U]\in\cS_i^-} \frac{\bra{[U]}\big(X_{m,n}^{(i)}Y_{m,n}^{(i)}\big)^{p+1}\ket{\vec{1}_i}}{\bra{[U]}\big(X_{m,n}^{(i)}Y_{m,n}^{(i)}\big)^{p}\ket{\vec{1}_i}}}
    .
  \end{equation}
  Then, $(C_{m,n;p}^{(i)})_{p\ge 0}$ is non-increasing, and
  \begin{equation}
    \label{eqn: CW corollary}
    \ell_3
    \le \frac{1}{mn}\log\rho\big(\widehat{T}_{m,n}\big)
    =\lim_{p\to\infty} \cbraces[\bigg]{\max_{0\le i\le\floor{mn/2}}C_{m,n;p}^{(i)}}
    =\inf_{p\ge0} \cbraces[\bigg]{\max_{0\le i\le\floor{mn/2}}C_{m,n;p}^{(i)}}.
  \end{equation}
\end{corollary}

\begin{proof}
  By construction, the matrix $X_{m,n}^{(i)}Y_{m,n}^{(i)}$ is nonnegative.
  To see that it has positive diagonal, note that, for all $U\subseteq V_{m,n}$ with $\abs{U}=i$,
  \begin{equation}
    \bra{[U]} X_{m,n}^{(i)}Y_{m,n}^{(i)} \ket{[U]}
    \ge \bra{[U]} X_{m,n}^{(i)} \ket{[V_{m,n}\setminus U]}
    \bra{[V_{m,n}\setminus U]} Y_{m,n}^{(i)} \ket{[U]}
    \ge 1\cdot 1
    =1.
  \end{equation} 
  Now, the inequality in~\eqref{eqn: CW corollary} follows from Theorem~\ref{thm: bound} and Proposition~\ref{prop: compression preserves spectral radius}, and the remaining conclusions follow from~\eqref{eqn: spectral radius decomposition} by applying Proposition~\ref{prop: iterated CW} with $A=X_{m,n}^{(i)}Y_{m,n}^{(i)}$ and $\vec{x}=\ket{\vec{1}_i}$ for each $0\le i\le\floor{mn/2}$.  
\end{proof}

The main advantage of Corollary~\ref{cor: CW corollary} is that the row sums in~\eqref{eqn: CW upper bound} can be computed by \textit{dynamic programming}, using the identity
\begin{equation}
  \begin{multlined}
    \bra{[U]} \big(X_{m,n}^{(i)}Y_{m,n}^{(i)}\big)^{p+1} \ket{\vec{1}_i}
    \\
    =\sum_{[V]\in\cS_i^+}\sum_{[W]\in\cS_i^-} \bra{[U]} X_{m,n}^{(i)} \ket{[V]} \bra{[V]} Y_{m,n}^{(i)} \ket{[W]} \bra{[W]} \big(X_{m,n}^{(i)}Y_{m,n}^{(i)}\big)^{p} \ket{\vec{1}_i},
  \end{multlined}
\end{equation}
for $0\le i\le \floor{mn/2}$ and $p\ge 0$.
In practice, we proceed as follows.
For each $0\le i\le \floor{mn/2}$:
\begin{enumerate}
  \item Maintain two vectors $\ket{\text{even},i}$ and $\ket{\text{odd},i}$, indexed respectively by $\cS_i^-$ and $\cS_i^+$, and with $\ket{\text{even},i}$ set to $\ket{\vec{1}_i}$ initially.
  At each step $p\ge 0$, we first overwrite $\ket{\text{odd},i}=Y_{m,n}^{(i)}\ket{\text{even},i}$, and then overwrite $\ket{\text{even},i}=X_{m,n}^{(i)}\ket{\text{odd},i}$.
  During the second update, each ratio 
  \begin{equation}
    \frac{\bra{[U]}X_{m,n}^{(i)}\ket{\text{odd},i}}{\bra{[U]}\ket{\text{even},i}}
  \end{equation}
  is compared to the \textit{running} maximum and the running maximum updated if necessary, before the entry $\bra{[U]}\ket{\text{even},i}$ is overwritten.
  Once the second update is complete, the running maximum is equal to the maximum in~\eqref{eqn: CW upper bound} and used to compute $C_{m,n;p}^{(i)}$.
  The running maximum is reset to $0$ before the next iteration.

  \item To compute the product $Y_{m,n}^{(i)}\ket{\vec{\text{even}},i}$, we iterate over the orbits $[V]\in \cS_i^+$.
  For each orbit, we fix a representative $V\in[V]$, and iterate over all mappings $\alpha:V\to D$ such that the mapping $\vec{v}\mapsto\vec{v}+\alpha(\vec{v})$ is injective.
  For each such mapping $\alpha$, write $V_\alpha:=\set{\vec{v}+\alpha(\vec{v})\mid\vec{v}\in V}$, and add 
  \begin{equation}
    \bra{[V_{m,n}\setminus V_\alpha]}\ket{\text{even},i}
  \end{equation}
  to the \textit{running} sum.
  Once the iteration over $\alpha$ is complete, the running sum is stored in $\bra{[V]}\ket{\text{odd},i}$.
  The other product $X_{m,n}^{(i)}\ket{\text{odd},i}$ is computed analogously.
\end{enumerate}
Thus, we never need to store any entry of $\widehat{T}_{m,n}$, at the cost of having to re-enumerate the relevant matchings for each matrix-vector multiplication.

If multiple computers are available, the sectors over $0\le i\le\floor{mn/2}$ may be processed in parallel; on a single computer, the decomposition above instead serves as a further dimensional reduction.
This is finer than the parity-based decomposition used by Ciucu~\cite[Section~3]{ciucu1998improved} and Lundow~\cite[Section~4.3]{lundow2001compression}, even after accounting for the latter's one additional split, because our decomposition yields $\floor{mn/2}+1$ sectors, rather than only two or three.

\section{Results}

Our numerical results are summarized in Table~\ref{tab: results}. 
All bounds were obtained from a Java implementation of Corollary~\ref{cor: CW corollary}, using $p=4$ except for $(m,n)=(5,5)$, for which $p=2$ was used to reduce the computation time.
All vector updates used arbitrary-precision integer arithmetic, and ratios were compared by exact cross multiplication. 
The sector decomposition described in Section~\ref{sec: matrix-free implementation} was used solely as an additional dimensional reduction, with the sectors processed sequentially rather than in parallel.
All computations were performed on a 64-bit desktop system equipped with an Intel Core i7-14700F processor (2.10 GHz) and 16 GB of RAM.
The computation for $(m,n)=(5,5)$ required the longest runtime of approximately 16.5 hours.

Theorem~\ref{thm: main} follows from the entry $(m,n)=(5,5)$. 
We also observe the following phenomenon: among pairs $(m,n)$ with comparable products, the bound is generally sharper when $m$ and $n$ are closer to one another; we do not have a theoretical explanation of this behavior.

Finally, we note that the method described in this paper can be generalized without modification to hypercubic lattices in higher dimensions, though we did not pursue this direction since, as Ciucu already noted~\cite[Section~4]{ciucu1998improved}, the community's interest is primarily in the three-dimensional case.

\begin{table}[t]
  \centering
  \begin{tabular}{ccccc}
    \toprule
    $m$ & \multicolumn{4}{c}{$n$} \\
    \cmidrule(lr){2-5}
     & 2 & 3 & 4 & 5 \\
    \midrule
    2 & 0.486478 & - & - & - \\
    3 & 0.471408 & 0.458558 & - & - \\
    4 & 0.465970 & 0.457380 & 0.454673 & - \\
    5 & 0.463401 & 0.454663 & 0.453365 & 0.452130 \\
    6 & 0.462010 & 0.454618 &  &  \\
    \bottomrule
  \end{tabular}
  \caption{Upper bounds on $\ell_3$ obtained from Corollary~\ref{cor: CW corollary}. A dash indicates that an entry is omitted by symmetry under exchanging $m$ and $n$; blank entries were not computed. Every entry is rounded upward to six decimal places.}
  \label{tab: results}
\end{table}

\section*{Acknowledgments} 

The author thanks Minhao Bai for helpful discussions about Ciucu's paper and Prof. Joel Lebowitz for summer support.
The author also thanks P{\'e}ter Csikv{\'a}ri for making us aware of his improved lower bound on $\ell_3$.
The idea to consider diagonal transfer layers and the connection with Csikv{\'a}ri's matching polynomial inequality~\cite{csikvari2017lower} are due to GPT-5.6 Sol Ultra.

\bibliographystyle{plain}
\bibliography{bibliography.bib}

\begin{thebibliography}{10}

\bibitem{barvinok2016combinatorics}
Alexander Barvinok.
\newblock {\em Combinatorics and complexity of partition functions}, volume~30 of {\em Algorithms and Combinatorics}.
\newblock Springer, 2016.

\bibitem{chandgotia2023large}
Nishant Chandgotia, Scott Sheffield, and Catherine Wolfram.
\newblock {Large deviations for the 3D dimer model}.
\newblock {\em arXiv preprint arXiv:2304.08468}, 2023.

\bibitem{chen2019transfer}
Zhibing Chen, Mordecai~J Golin, and Xuerong Yong.
\newblock The transfer matrices and the capacity of the 2-dimensional $(1,\infty)$-runlength limited constraint.
\newblock {\em Discrete Mathematics}, 342(4):975--987, 2019.

\bibitem{ciucu1998improved}
Mihai Ciucu.
\newblock An improved upper bound for the 3-dimensional dimer problem.
\newblock {\em Duke Mathematical Journal}, 94(1):1--11, 1998.

\bibitem{csikvari2016matchings}
P{\'e}ter Csikv{\'a}ri.
\newblock Matchings in vertex-transitive bipartite graphs.
\newblock {\em Israel Journal of Mathematics}, 215(1):99--134, 2016.

\bibitem{csikvari2017lower}
P{\'e}ter Csikv{\'a}ri.
\newblock {Lower matching conjecture, and a new proof of Schrijver's and Gurvits's theorems}.
\newblock {\em Journal of the European Mathematical Society}, 19(6):1811--1844, 2017.

\bibitem{friedland2005theory}
Shmuel Friedland and Uri~N Peled.
\newblock Theory of computation of multidimensional entropy with an application to the monomer--dimer problem.
\newblock {\em Advances in Applied Mathematics}, 34(3):486--522, 2005.

\bibitem{friedland1980growth}
Shmuel Friedland and Hans Schneider.
\newblock The growth of powers of a nonnegative matrix.
\newblock {\em SIAM Journal on Algebraic Discrete Methods}, 1(2):185--200, 1980.

\bibitem{hammersley1966existence}
John~M Hammersley.
\newblock {Existence theorems and Monte Carlo methods for the monomer-dimer problem}.
\newblock {\em Research papers in statistics: Festschrift for J. Neyman}, pages 125--146, 1966.

\bibitem{horn2012matrix}
Roger~A Horn and Charles~R Johnson.
\newblock {\em Matrix Analysis}.
\newblock Cambridge University Press, 2012.

\bibitem{kasteleyn1961statistics}
Pieter~W Kasteleyn.
\newblock {The statistics of dimers on a lattice: I. The number of dimer arrangements on a quadratic lattice}.
\newblock {\em Physica}, 27(12):1209--1225, 1961.

\bibitem{Klivans_Saldanha_2026}
Caroline~J Klivans and Nicolau Saldanha.
\newblock {Domino tilings beyond 2D}.
\newblock In {\em Surveys in Combinatorics 2026}, London Mathematical Society Lecture Note Series, pages 267--292. Cambridge University Press, 2026.

\bibitem{lelarge2015counting}
Marc Lelarge.
\newblock Counting matchings in irregular bipartite graphs and random lifts.
\newblock {\em arXiv preprint arXiv:1507.04739}, 2015.

\bibitem{lundow2001compression}
Per~H{\aa}kan Lundow.
\newblock Compression of transfer matrices.
\newblock {\em Discrete Mathematics}, 231(1-3):321--329, 2001.

\bibitem{lundow2008exact}
Per~H{\aa}kan Lundow and Klas Markstr{\"o}m.
\newblock Exact and approximate compression of transfer matrices for graph homomorphisms.
\newblock {\em LMS Journal of Computation and Mathematics}, 11:1--14, 2008.

\bibitem{minc1983theory}
Henryk Minc.
\newblock Theory of permanents 1978--1981.
\newblock {\em Linear and Multilinear Algebra}, 12(4):227--263, 1983.

\bibitem{nagle1966new}
John~F Nagle.
\newblock New series-expansion method for the dimer problem.
\newblock {\em Physical Review}, 152(1):190--197, 1966.

\bibitem{otsuka2022residual}
Takahiro Otsuka and Yasuhiro Akutsu.
\newblock Residual entropy of the cubic dimer model: The higher order tensor renormalization group approach.
\newblock {\em Journal of the Physical Society of Japan}, 91(5):055003, 2022.

\bibitem{schrijver1998counting}
Alexander Schrijver.
\newblock Counting 1-factors in regular bipartite graphs.
\newblock {\em Journal of Combinatorial Theory, Series B}, 72(1):122--135, 1998.

\bibitem{schrijver1980lower}
Alexander Schrijver and WG~Valiant.
\newblock On lower bounds for permanents.
\newblock {\em Indagationes Mathematicae (Proceedings)}, 83(4):425--427, 1980.

\bibitem{taggi2022uniformly}
Lorenzo Taggi.
\newblock {Uniformly Positive Correlations in the Dimer Model and Macroscopic Interacting Self-Avoiding Walk in $\mathbb{Z}^d$, $d\ge 3$}.
\newblock {\em Communications on Pure and Applied Mathematics}, 75(6):1183--1236, 2022.

\bibitem{temperley1961dimer}
Harold~NV Temperley and Michael~E Fisher.
\newblock Dimer problem in statistical mechanics---an exact result.
\newblock {\em Philosophical Magazine}, 6(68):1061--1063, 1961.

\bibitem{vanderstraeten2018residual}
Laurens Vanderstraeten, Bram Vanhecke, and Frank Verstraete.
\newblock Residual entropies for three-dimensional frustrated spin systems with tensor networks.
\newblock {\em Physical Review E}, 98(4):042145, 2018.

\end{thebibliography}

\appendix

\section{Proof of Proposition~\ref{prop: iterated CW}}
\label{appx: iterated CW}

In this appendix, we prove the iterated Collatz--Wielandt formula stated in Proposition~\ref{prop: iterated CW}. 
Monotonicity is elementary, while convergence to the spectral radius is established through a blockwise analysis in Frobenius normal form using the asymptotic results of Friedland--Schneider~\cite{friedland1980growth}.

Since $A$ has positive diagonal and $\vec{x}>\vec{0}$, we have $(A^q\vec{x})_i\ge a_{ii}^q x_i>0$, so $C_q$ is well-defined.
By definition, $A^{q+1}\vec{x}\le C_q A^q\vec{x}$; multiplying both sides by $A$ gives $A^{q+2}\vec{x}\le C_q A^{q+1}\vec{x}$, and hence $C_{q+1}\le C_q$.
The Collatz--Wielandt formula~\cite[Theorem~8.1.26]{horn2012matrix} yields $\rho(A)\le C_q$ for all $q\ge 0$, so that
\begin{equation}
  \rho(A)\le\liminf_{q\to\infty}C_q.
\end{equation}
It remains to prove that
\begin{equation}
  \label{eqn: hardest part of proof}
  \limsup_{q\to\infty}C_q\le\rho(A),
\end{equation}
which is the hardest part of the proof.

We import some language from Friedland--Schneider~\cite{friedland1980growth}.
Let $A$ be a nonnegative square matrix in Frobenius normal form~\eqref{eqn: Frobenius normal form}, normalized so that $\rho(A)=1$.
For $1\le\alpha\le\nu$, let $I_\alpha$ denote the index set such that $A_{\alpha\alpha}=A[I_\alpha,I_\alpha]$.
The \textbf{period} of $A$ is the smallest positive integer $q$ such that every eigenvalue of $A$ of maximum modulus is a $q$th root of unity.
Construct a digraph $G(A)$ on the vertex set $\set{1,\dots,\nu}$ which contains the edge $(\alpha,\beta)$ if and only if $A_{\alpha\beta}\ne 0$; in particular, $(\alpha,\alpha)\in G(A)$ for each $\alpha$.
A \textbf{path} in $G(A)$ is either a single vertex or a sequence of vertices $(\alpha_0,\dots,\alpha_s)$, $s\ge 1$, such that $1\le\alpha_0<\dots<\alpha_s\le\nu$ and $(\alpha_{t-1},\alpha_t)\in G(A)$ for $1\le t\le s$.
Declare a vertex $\alpha$ to be \textbf{singular} if $\rho(A_{\alpha\alpha})=\rho(A)$.
For $1\le\alpha\le\beta\le\nu$, define $k(\alpha,\beta)$, the \textbf{singular distance} from $\alpha$ to $\beta$, as the maximum number of singular vertices on a path from $\alpha$ to $\beta$ in $G(A)$ \textit{minus one}; if no such path exists, set $k(\alpha,\beta):=-\infty$.
Moreover, a path attaining the maximum number of singular vertices is said to be \textbf{maximal}.
For a maximal path $\pi$, define $q(\pi)$ as the greatest common divisor of the periods of the singular vertices on $\pi$.
Define $q(\alpha,\beta)$, the \textbf{local period} of $(\alpha,\beta)$, as the least common multiple of $q(\pi)$ over all maximal paths $\pi$ from $\alpha$ to $\beta$; if $k(\alpha,\beta)<0$, set $q(\alpha,\beta):=1$.
By extension, if $i\in I_\alpha$ and $j\in I_\beta$, where $1\le\alpha\le\beta\le\nu$, define the singular distance from $i$ to $j$ as $k[i,j]:=k(\alpha,\beta)$ and the local period of $(i,j)$ as $q[i,j]:=q(\alpha,\beta)$.

We are now ready to prove~\eqref{eqn: hardest part of proof}.

First, we put $A$ in Frobenius normal form~\cite[Problem~8.3.P8]{horn2012matrix}.
Let $P$ be a permutation matrix such that $F=P^TAP$ is in Frobenius normal form, and let $\vec{y}=P^T\vec{x}$.
Since $F^q\vec{y}=P^TA^q\vec{x}$, we have 
\begin{equation}
  \max_{1\le i\le n}\frac{(A^{q+1}\vec{x})_i}{(A^q\vec{x})_i}
  =\max_{1\le i\le n}\frac{(F^{q+1}\vec{y})_i}{(F^q\vec{y})_i}.
\end{equation}
Moreover, $F$ has positive diagonal and $\rho(F)=\rho(A)$.
Therefore, we may replace $(A,\vec{x})$ by $(F,\vec{y})$ and assume without loss of generality that $A$ is in Frobenius normal form:
\begin{equation}
  \label{eqn: Frobenius normal form}
  A=
  \begin{bmatrix}
    A_{11} & A_{12} & \cdots & A_{1\nu} \\
    0 & A_{22} & \cdots & A_{2\nu} \\
    \vdots & \vdots & \ddots & \vdots \\
    0 & 0 & \cdots & A_{\nu\nu}
  \end{bmatrix}
  ,
\end{equation}
where each $A_{\alpha\alpha}$ is irreducible~\cite[Definition~6.2.22]{horn2012matrix}; in fact, since $A$ has positive diagonal, each $A_{\alpha\alpha}$ is primitive by~\cite[Lemma~8.5.4]{horn2012matrix}.
Hence, by the Perron--Frobenius theorem~\cite[Theorem~8.4.4]{horn2012matrix}, each $A_{\alpha\alpha}$ has $\rho(A_{\alpha\alpha})>0$ as its unique eigenvalue of maximum modulus.

For $1\le i\le n$, define the index set
\begin{equation}
  S_i:=\set[\big]{1\le j\le n\mid\text{there exists }m\ge 1\text{ such that }(A^m)_{ij}>0}
\end{equation}
and let $A_i:=A[S_i,S_i]$.

\begin{lemma}
  \label{lem: restriction to S_i}
  Let $\vec{x}^{(i)}:=\vec{x}[S_i]$. 
  For all $q\ge 0$, $(A^q\vec{x})_i=(A_i^q\vec{x}^{(i)})_i$.
\end{lemma}

\begin{proof}
  It suffices to show that $A[S_i,S_i^c]=0$.
  Suppose that $j\in S_i$ and $a_{jk}>0$.
  Choosing $m\ge 1$ such that $(A^m)_{ij}>0$, we have that $(A^{m+1})_{ik}\ge (A^m)_{ij} a_{jk}>0$, so $k\in S_i$.
\end{proof}

Furthermore, $A_i$ consists of blocks from the Frobenius normal form~\eqref{eqn: Frobenius normal form}:

\begin{lemma}
  \label{lem: S_i contains I_alpha}
  If $I_\alpha\cap S_i\ne\emptyset$, then $I_\alpha\subseteq S_i$.
\end{lemma}

\begin{proof}
  Let $k\in I_\alpha\cap S_i$ and $m_1\ge 1$ be such that $(A^{m_1})_{ik}>0$.
  Suppose that $j\in I_\alpha$.
  Since $A_{\alpha\alpha}$ is primitive, there exists $m_2\ge 1$ such that $A_{\alpha\alpha}^{m_2}>0$~\cite[Theorem~8.5.2]{horn2012matrix}, so, in particular, $(A^{m_2})_{kj}>0$.
  Then, $(A^{m_1+m_2})_{ij} \ge (A^{m_1})_{ik} (A^{m_2})_{kj} > 0$, so $j\in S_i$.
\end{proof}

We also have control on the spectral radius of $A_i$:

\begin{lemma}
  \label{lem: spectral radius of A_i}
  For $1\le i\le n$, $0<\rho(A_i)\le\rho(A)$.
  In particular, $\max_{1\le i\le n}\rho(A_i)=\rho(A)$.
\end{lemma}

\begin{proof}
  Since $A$ has positive diagonal, so does $A_i$, and hence $\rho(A_i)>0$ by~\cite[Corollary~8.1.20(c)]{horn2012matrix}.
  Since $A_i$ is a principal submatrix of $A$, it follows from~\cite[Corollary 8.1.20(a)]{horn2012matrix} that $\rho(A_i)\le\rho(A)$.
  Lastly, suppose that $1\le\alpha\le\nu$ is such that $\rho(A_{\alpha\alpha})=\rho(A)$.
  Then, any $1\le i\le n$ such that $i\in I_\alpha$ satisfies $\rho(A_i)\ge\rho(A_{\alpha\alpha})=\rho(A)$ by Lemma~\ref{lem: S_i contains I_alpha} and~\cite[Corollary 8.1.20(a)]{horn2012matrix}.
\end{proof}

By Lemma~\ref{lem: S_i contains I_alpha}, $A_i$ is already in Frobenius normal form, and its irreducible diagonal blocks are among the blocks $A_{\alpha\alpha}$ in~\eqref{eqn: Frobenius normal form}.
In particular, they are all primitive, so every singular block of $B_i:=\rho(A_i)^{-1}A_i$, of which there is at least one, has period one.
For $j\in S_i$, let $k_i[j]$ denote the singular distance from $i$ to $j$ for $B_i$, $q_i[j]$ the local period of $(i,j)$ in $B_i$, and $d_i:=\max_{j\in S_i}k_i[j]$.
For $j\in S_i$, we have that $k_i[j]>-\infty$ by the definition of $S_i$, and $q_i[j]=1$.
Moreover, by taking $j\in S_i\cap I_\alpha$ where $1\le\alpha\le\nu$ is such that $\rho(A_i)^{-1}A_{\alpha\alpha}$ is a singular block of $B_i$, we see that $d_i\ge 0$. 
By~\cite[Theorem~5.10]{friedland1980growth}, if $k_i[j]\ge 0$, then there exists a constant $c_{ij}>0$ such that
\begin{equation}
  \lim_{r\to\infty} \frac{(B_i^r)_{ij}}{q^{k_i[j]}}
  =c_{ij};
\end{equation}
and if $k_i[j]=-1$, then there exists $0<\rho_{ij}<1$ such that
\begin{equation}
  \lim_{r\to\infty} \frac{(B_i^r)_{ij}}{\rho_{ij}^{r}}
  =0.
\end{equation}
Thus,
\begin{equation}
  \lim_{r\to\infty} \frac{(B_i^r\vec{x}^{(i)})_i}{r^{d_i}}
  =\lim_{r\to\infty} \sum_{j\in S_i} \frac{(B_i^r)_{ij}x_j}{r^{k_i[j]} r^{d_i-k_i[j]}}
  =\sum_{\substack{j\in S_i\\k_i[j]=d_i}} c_{ij}x_j
  >0,
\end{equation}
and so
\begin{equation}
  \label{eqn: consecutive ratio limit}
  \lim_{r\to\infty} \frac{(A_i^{r+1}\vec{x}^{(i)})_i}{\rho(A_i) (A_i^{r}\vec{x}^{(i)})_i}
  =\lim_{r\to\infty} \frac{(B_i^{r+1}\vec{x}^{(i)})_i}{(B_i^{r}\vec{x}^{(i)})_i}
  =\lim_{r\to\infty} \frac{(B_i^{r+1}\vec{x}^{(i)})_i r^{d_i}}{(B_i^{r}\vec{x}^{(i)})_i (r+1)^{d_i}}
  =1.
\end{equation}
We conclude that
\begin{align}
  \limsup_{r\to\infty} C_r
  {}& \le\lim_{r\to\infty} \max_{1\le i\le n}\frac{(A^{r+1}\vec{x})_i}{(A^r\vec{x})_i}
  {}&&\text{by~\eqref{eqn: definition of C_q}}
  \nonumber
  \\
  {}& =\max_{1\le i\le n}\lim_{r\to\infty}\frac{(A^{r+1}_i\vec{x}^{(i)})_i}{(A^r_i\vec{x}^{(i)})_i}
  {}&&\text{by Lemma~\ref{lem: restriction to S_i}}
  \nonumber
  \\
  {}& =\max_{1\le i\le n}\rho(A_i)
  {}&&\text{by~\eqref{eqn: consecutive ratio limit}}
  \nonumber
  \\
  {}& =\rho(A)
  {}&&\text{by Lemma~\ref{lem: spectral radius of A_i}},
\end{align}
as desired.

\end{document}